\documentclass[10pt,reqno]{amsart}
\usepackage{amsmath,amssymb,latexsym,soul,cite,mathrsfs}
\usepackage{color,enumitem,graphicx}

\usepackage[colorlinks=true,urlcolor=blue,
citecolor=red,linkcolor=blue,linktocpage,pdfpagelabels,
bookmarksnumbered,bookmarksopen]{hyperref}
\usepackage[english]{babel}
\usepackage{enumitem}

\usepackage[left=2.0cm,right=2.0cm,top=2.5cm,bottom=2.5cm]{geometry}

\makeatletter
\newcommand{\leqnomode}{\tagsleft@true}
\newcommand{\reqnomode}{\tagsleft@false}
\makeatother

\numberwithin{equation}{section}

\newtheorem{theorem}{Theorem}[section]
\newtheorem{lemma}[theorem]{Lemma}

\newtheorem{proposition}[theorem]{Proposition}
\newtheorem{cor}[theorem]{Corollary}
\newtheorem{remark}[theorem]{Remark}
\newtheorem{definition}[theorem]{Definition}

\renewcommand{\rightarrow}{\to}

\providecommand{\ln}{\mathop{\rm ln}\nolimits}

\newcommand{\dive}{\mathrm{div}}

\newcommand{\w}{W_0^{1,\Phi}(\Omega)}

\title[Revised regularity results for quasilinear elliptic problems]{Revised regularity results for quasilinear elliptic problems driven by the $\Phi$-Laplacian operator}

\author[E.D. Silva]{E. D. Silva}
\author[M.L. Carvalho]{M. L. Carvalho}
\author[J.C. \ de Albuquerque]{J. C. de Albuquerque}

\address[E.D. Silva]{Department of Mathematics, Federal University of Goi\'{a}s}
\email{\href{mailto:eddomingos@hotmail.com}{eddomingos@hotmail.com}}

\address[M.L. Carvalho]{Department of Mathematics, Federal University of Goi\'{a}s
Federal University of Goi\'{a}s
\newline\indent
74001-970, Goi\'{a}s-GO, Brazil}
\email{\href{marcos_leandro_carvalho@ufg.br}{marcos$\_$leandro$\_$carvalho@ufg.br}}

\address[J.C. de~Albuquerque]{Department of Mathematics, Federal University of Goi\'{a}s}
\email{\href{mailto:joserre@gmail.com}{joserre@gmail.com}}

\thanks{Corresponding author: M. L. Carvalho.}

\thanks{Research supported in part by INCTmat/MCT/Brazil, CNPq and CAPES/Brazil. The authors was partially supported by Fapeg/CNpq grants 03/2015-PPP}


\begin{document}
	

\begin{abstract}
 	It is establish regularity results for weak solutions of quasilinear elliptic problems driven by the well known $\Phi$-Laplacian operator given by  
 	  \begin{equation*}
 	 \left\{\
 	 \begin{array}{cl}
 	 \displaystyle-\Delta_\Phi u= g(x,u), & \mbox{in}~\Omega,\\
 	 u=0, & \mbox{on}~\partial \Omega,
 	 \end{array}
 	 \right.
 	 \end{equation*}
 	where $\Delta_{\Phi}u :=\mbox{div}(\phi(|\nabla u|)\nabla u)$ and $\Omega\subset\mathbb{R}^{N}, N \geq 2,$ is a bounded domain with smooth boundary $\partial\Omega$. Our work concerns on nonlinearities $g$ which can be homogeneous or non-homogeneous. For the homogeneous case we consider an existence result together with a regularity result proving that any weak solution remains bounded.  Furthermore, for the non-homogeneous case, the nonlinear term $g$ can be subcritical or critical proving also that any weak solution is bounded. The proofs are based on Moser's iteration in Orclicz and Orlicz-Sobolev spaces. 
\end{abstract}

\maketitle


\section{Introduction}	
In this work we establish regularity results for weak solutions of the quasilinear elliptic problems driven by the $\Phi$-Laplacian operator described in the following form
 \begin{equation}\label{prob}
 \left\{\
 \begin{array}{cl}
 \displaystyle-\Delta_\Phi u= g(x,u), & \mbox{in}~\Omega,\\
 u=0, & \mbox{on}~\partial \Omega,
 \end{array}
 \right. \tag{$P$}
 \end{equation}
where $\Omega\subset\mathbb{R}^{N}$ is bounded domain with smooth boundary $\partial\Omega$, $\Delta_{\Phi} u:=\mbox{div}(\phi(|\nabla u|)\nabla u)$ is the $\Phi$-Laplacian operator and $g:\Omega\times\mathbb{R}\rightarrow\mathbb{R}$ is a Carath\'eodory function satisfying suitable assumptions. Throughout this work we shall consider $\Phi:\mathbb{R}\rightarrow\mathbb{R}$ an even function defined by
\begin{equation}\label{phi}
\Phi(t)=\int_0^ts\phi(s)\,\mathrm{d} s.
\end{equation}
The function $\phi:\mathbb{R}\rightarrow\mathbb{R}$ is a $C^1$-function satisfying the following assumptions:
\begin{itemize}
	\item[($\phi_1$)] $t\phi(t)\mapsto 0$, as $t\mapsto 0$ and $t\phi(t)\mapsto\infty$, as $t\mapsto\infty$;
	\item[($\phi_2$)] $t\phi(t)$ is strictly increasing in $(0, \infty)$;
	\item[($\phi_3$)] there exist $\ell\in[1,N)$ and $m\in(1,N)$ such that
	\[
	\ell-1\leq \frac {(t\phi(t))^\prime}{\phi(t)}\leq m-1<\ell^*-1, \quad \mbox{for all} \hspace{0,2cm} t>0.
	\]
\end{itemize}
Due to the nature of the non-homogeneous differential operator $\Phi$-Laplacian, we shall work in the framework of Orlicz and Orlicz-Sobolev spaces. For the reader's convenience, we provide an Appendix with a brief revision on the Orlicz space setting. It is worthwhile
to mention that the Orlicz space $L_{\Phi}(\Omega)$ is a generalization of the Lebesgue space $L^{p}(\Omega)$. It is well known that the Orlicz-Sobolev space $W^{1,\Phi}(\Omega)$ is a generalization of the classical Sobolev space $W^{1,p}(\Omega)$. Hence, several properties of the Sobolev spaces have been extended to Orlicz-Sobolev spaces. The main interest regarding Orlicz-Sobolev spaces is motivated by their applicability in many fields of mathematics, such as partial differential equations, calculus of variations, non-linear potential theory, differential geometry, geometric function theory, the theory of quasiconformal mappings, probability theory, non-Newtonian fluids, image processing, among others, see \cite{adams,Fuk_2,Fuk_1,rao}. The class of problems introduced in \eqref{prob} is related with several branch of physics which are based on the nature of the non-homogeneous nonlinearity $\Phi$. For instance we cite the following examples:
\begin{itemize}
	\item[(i)] Nonlinear elasticity: $\Phi(t)=(1+t^{2})^{\gamma}-1$, $1<\gamma<N/(N-2)$;
	\item[(ii)] Plasticity: $\Phi(t)=t^{\alpha}(\log(1+t))^{\beta}$, $\alpha\geq1$, $\beta>0$;
	\item[(iii)] Non-Newtonian fluid: $\Phi(t)=\frac{1}{p}|t|^{p}$, for $p>1$;
	\item[(iv)] Plasma physics: $\Phi(t)=\frac{1}{p}|t|^{p}+\frac{1}{q}|t|^{q}$, where $1<p<q<N$ with $q\in(p,p^{*})$;
	\item[(v)] Generalized Newtonian fluids: $\Phi(t)=\int_{0}^{t}s^{1-\alpha}[\sinh^{-1}(s)]^{\beta}\,\mathrm{d}s$, $0\leq\alpha\leq 1$, $\beta>0$.
\end{itemize}
In the example (iii), the function $\Phi$ gives the so called $p$-Laplacian and Problem \eqref{prob} can be read as
  \begin{equation*}
 \left\{\
 \begin{array}{cl}
 \displaystyle-\Delta_{p} u = g(x,u), & \mbox{in}~\Omega,\\
 u=0, & \mbox{on}~\partial \Omega,
 \end{array}
 \right.
 \end{equation*}
In similar way, in the example (iv), the function $\Phi$ provides the named $(p,q)$-Laplacian operator and  Problem \eqref{prob} can be rewritten in the following form
   \begin{equation*}
 \left\{\
 \begin{array}{cl}
 \displaystyle-\Delta_{p} u-\Delta_{q} u = g(x,u), & \mbox{in}~\Omega,\\
 u=0, & \mbox{on}~\partial \Omega.
 \end{array}
 \right.
 \end{equation*}
 It is worthwhile to recall that $\Phi$ satisfies the so called $\Delta_2$-condition whenever  
 \begin{equation*}
 \Phi(2 t) \leq C \Phi(t),\ t \geq t_{0},
 \end{equation*}
 holds true for some $C > 0$ and for some $t_{0} \geq 0$. In short, we write $\Phi \in \Delta_2$. One feature on this work is to consider regularity results for quasilinear elliptic problem driven by the $\Phi-$Laplacian operator where the so called $\Delta_{2}$-condition is not satisfied for $\tilde{\Phi}$, that is, the conjugate function defined by
 $$
 \widetilde{\Phi}(t) = \displaystyle \max_{s \geq 0} \{ts - \Phi(s) \},~ t \geq 0,
 $$
 does not verifies the $\Delta_2$-condition. It is important to emphasize that the $\Phi$- Laplacian operator is not homogeneous which bring us several difficulties in order to get the boundedness of a weak solution to the elliptic Problem \eqref{prob}. Moreover, the Orlicz space can be different from any usual Lebesgue spaces, for instance, when $\Phi(t)=t^{\alpha}(\log(1+t))^{\beta}$ for some $\alpha \geq 1$ and $\beta > 0$. For more details about non-homogeneous differential operators with different types of nonlinearity $\Phi$ we refer the readers to \cite{dibene,rad1,rad2,rad3,Fuk_2} and references therein. 

There is a huge bibliography concerned on regularity results for problems related to \eqref{prob}. We refer the readers to interesting works \cite{pucci2,cianchi,FG,lieberman,fang,dibene}. There are many applications of regularity theory for quasilinear elliptic problem defined on bounded domains. For instance, an application of our regularity results is a version of the strong maximum principle for the quasilinear elliptic problems given by Problem \eqref{prob}, see Theorem \ref{smp} ahead. Another interesting application arises from the study of existence of solutions which satisfies a multivalued elliptic equation in an ``almost everywhere" sense. More specifically, let $u\in W^{1,\Phi}_{loc}(\Omega)$ be a solution of Problem \eqref{prob} in such way that $[\nabla u=0]$ is the associated singular set. The regularity of the solution may be used to prove that the Lebesgue measure of the singular set is null. This type of result was proved, for example, by H. Lou \cite{Lou} only for the case $\Phi(t):=t^p/p,~1<p<\infty$. By using this fact, one can conclude that a solution for a multivalued problem satisfies the equation almost everywhere.  The same argument can be used for the $\Phi$-Laplacian operator thanks to the fact that any weak solution for the Problem \eqref{prob} remains bounded. This assertion is the most feature in the present work.  It is also important to emphasize that any weak solution for the Problem \eqref{prob} can be not a local minimum for the energy functional associated to the problem \eqref{prob}. There exist several regularity results concerning only for local minimizers for a suitable energy functional. On this subject we refer the readers to \cite{Alves1} and references therein. For further results concerning on related results for quasilinear elliptic problems involving nonhomogeneous operators we refer the reader to  Fiscella and Pucci \cite{FP}. In the present work is not needed to assume that a weak solution for Problem \eqref{prob} is a minimum for the energy functional. Then our work complements/extends the aforementioned works.  

The main contribution on this work is to guarantee some regularity results for quasilinear elliptic equations driven by the $\Phi$-Laplacian operator for the homogeneous and non-homogeneous case. The main feature is to ensure that any weak solution for the Problem \eqref{prob} are necessarily bounded. More precisely, we shall consider quasilinear elliptic problem given by the $\Phi$-Laplacian operator showing regularity results taking into account a truncation technique together with the Moser's iteration. For the homogeneous case, we study regularity of solutions for following quasilinear elliptic problem
  \begin{equation}\label{probhomo}
 \left\{\
 \begin{array}{cl}
 \displaystyle-\Delta_\Phi u= f(x), & \mbox{in}~\Omega,\\
 u= 0, & \mbox{on}~\partial \Omega,
 \end{array}
 \right. \tag{$P_h$}
 \end{equation}
where $f\in L^{q}(\Omega)$, with $q>N/\ell$ and $\ell>1$. It is worthwhile to recall that $u\in W^{1,\Phi}_{0}(\Omega)$ is said to be a weak solution for the quasilinear elliptic Problem~\eqref{probhomo} if there holds
\[
\int_{\Omega}\Phi(|\nabla u|)\nabla u\nabla v\,\mathrm{d}x=\int_{\Omega}f(x)v\,\mathrm{d}x, \quad \mbox{for all} \hspace{0,2cm} v\in W^{1,\Phi}_{0}(\Omega).
\]

\begin{definition}
	Let  $\Phi, \Psi$ be two $N$-functions. We say that $\Phi$ and $\Psi$ are equivalent, in short 
	$\Phi \approx \Psi$, when there exist $c_{1}, c_{2} > 0$ in such way that $c_{1} \Psi(t) \leq \Phi(t) \leq c_{2} \Psi(t)$ for any $ t \geq t_{0}$ and for some $t_{0} \geq 0$. Moreover, we write $\Phi \not\approx \Psi$ whenever $\Phi$ and $\Psi$ are not equivalent.
\end{definition}

Our first main result can be stated as follows:

\begin{theorem}\label{A}
	Suppose that $(\phi_{1})-(\phi_{3})$ hold with $\ell\in(1,N)$. Then, Problem \eqref{probhomo} possesses a positive solution $u\in \w$. Moreover, assume that one of the following hypotheses holds:
	 \begin{itemize}
	 	\item[(i)] $\Phi\not\approx t^{m}$ and $f\in L^{q}(\Omega)$ is nonnegative with $q>N/\ell$;
	 	\item[(ii)] $\Phi\approx t^{m}$ and $f\in L^{q}(\Omega)$ is nonnegative with $q>N/m$.
	 \end{itemize}	
    Then the weak solution $u$ for the problem \eqref{probhomo} belongs to $L^{\infty}(\Omega)$.
\end{theorem}

\begin{remark}
	It is well known that Problem \eqref{probhomo} admits solution, see \cite{gossez-Czech,Fuk_1} and references therein. However, we give an alternative proof for the existence of solution by constructing a monotone sequence of solutions for truncated problems which converges to a solution of \eqref{probhomo}. This sequence will be used to define a suitable test function in the Moser iteration method. As a consequence we show that Problem \eqref{probhomo} admits a bounded solution. Precisely, thanks to the strictly monotonicity of the $\Phi-$Laplacian operator, the solution is unique. 
\end{remark}

We are also concerned with regularity results for the non-homogeneous problem \eqref{prob}. This case extends Theorem~\ref{A} in some directions. For instance, we study the regularity of solutions when $q=N/\ell$ in a suitable sense. Moreover, we consider a Carath\'eodory function $g:\Omega\times\mathbb{R}\rightarrow\mathbb{R}$ satisfying subcritical or critical growth. Let $\alpha\in\{\ell,m\}$ be a fixed number. For the subcritical case we suppose that 
 \begin{equation}\label{sub}
 |g(x,t)|\leq a(x)(1+|t|^{\alpha-1}), \quad \mbox{for all} \hspace{0,2cm} (x,t)\in \Omega\times\mathbb{R},
 \end{equation}
where $a\in L^{N/\alpha}(\Omega)$. It is also usual to consider the following subcritical behavior for $g$ given by 
\begin{equation}\label{ellsub}
|g(x,t)|\leq C(|t|^{\alpha-1}+|t|^{r-1}), \quad \mbox{for all} \hspace{0,2cm} (x,t)\in \Omega\times\mathbb{R}.
\end{equation}
where $\alpha < r < \alpha^*$. For the critical case, we assume that there exists $C>0$ such that
 \begin{equation}\label{ell}
 |g(x,t)|\leq C(|t|^{\alpha-1}+|t|^{\alpha^{*}-1}), \quad \mbox{for all} \hspace{0,2cm} (x,t)\in \Omega\times\mathbb{R}.
 \end{equation}
 
Now we state our main result regarding to the non-homogeneous case.

\begin{theorem}\label{B}
	Suppose that $(\phi_{1})-(\phi_{3})$ hold. Let $u\in W^{1,\Phi}_{0}(\Omega)$ be a weak solution for Problem~\eqref{prob}. Assume that \eqref{sub} holds true. Assume also that one of the following hypotheses is verified:
	\begin{itemize}
		\item[(i)] $\Phi\not\approx t^{m}$ and $\alpha=\ell$;		
		\item[(ii)] $\Phi\approx t^{m}$ and $\alpha=m$;
	\end{itemize}	
	Then the solution $u$ is in $L^{q}(\Omega)$ for all $q\in[1,\infty)$.
\end{theorem}
It is important to point out that Theorem \ref{B} can be viewed as a generalization of the well known celebrated result of Brezis-Kato \cite{BK}. As a consequence, using a critical or subcritical behavior for $g$, we can state the following regularity result:
\begin{cor}\label{cor1}
	Suppose that $(\phi_{1})-(\phi_{3})$ hold. Let $u\in W^{1,\Phi}_{0}(\Omega)$ be a weak solution for Problem~\eqref{prob}. Assume that \eqref{ellsub} or  \eqref{ell}  holds. Assume also that one of the following hypotheses is satisfied:
	 \begin{itemize}
	 	\item[(i)] $\Phi\not\approx t^{m}$ and $\alpha=\ell$;		
	 	\item[(ii)] $\Phi\approx t^{m}$ and $\alpha=m$;
	 \end{itemize}   
 Then the weak solution $u$ belongs to $L^{q}(\Omega)$ for all $q\in[1,\infty)$.
\end{cor}

In order to get our main result we consider also the where $g$ is subcritical or critical. Then we can show that any weak solution for the quasilinear elliptic problem \eqref{prob} is in $L^{\infty}(\Omega)$. This result can be state as follows:

\begin{theorem}\label{thrm}
	Suppose that $(\phi_{1})-(\phi_{3})$ hold. Let $u\in W^{1,\Phi}_{0}(\Omega)$ be a weak solution for Problem~\eqref{prob}.  Assume that \eqref{ellsub} or  \eqref{ell}  holds true. Assume also that one of the following hypotheses is satisfied:
	\begin{itemize}
		\item[(i)] $\Phi\not\approx t^{m}$ and $\alpha=\ell$;		
		\item[(ii)] $\Phi\approx t^{m}$ and $\alpha=m$;
	\end{itemize}   
 Then the weak solution $u$ is in $L^{\infty}(\Omega)$.    
\end{theorem}

As an application we can ensure that any nonnegative solution for the quasilinear elliptic problem \eqref{prob} is strictly positive. In other words, we can state the following Strong Maximum Principle for quasilinear elliptic equations driven by the $\Phi$-Laplacian operator as follows:
\begin{theorem}[Strong Maximum Principle]\label{smp}
	Suppose that $(\phi_{1})-(\phi_{3})$ hold. Let $u\in W^{1,\Phi}_{0}(\Omega)$ be a nonnegative weak solution of \eqref{prob}, where $g:\Omega\times\mathbb{R}\rightarrow\mathbb{R}$ satisfies \eqref{ellsub} or  \eqref{ell} with $\ell>1$. Moreover, suppose that there exists $\delta>0$ such that $g(x,s)>0$ for all $x\in\Omega$ and $s\in(0,\delta)$. Then we obtain that $u\in C^{1,\alpha}(\overline{\Omega})$ and $u>0$ in $\Omega$.
\end{theorem}

\begin{remark}
	We mention that our results remain true for more general quasilinear elliptic problems. For instance, we can consider the following class of problems 
	  \begin{equation*}
	 \left\{\
	 \begin{array}{cl}
	  -\dive{A}(x,u,\nabla u)= g(x,u, \nabla u), & \mbox{in}~\Omega,\\
	 u= 0, & \mbox{on}~\partial \Omega,
	 \end{array}
	 \right. 
	 \end{equation*}
	where $A:\Omega\times\mathbb{R}\times\mathbb{R}^{N}\rightarrow\mathbb{R}^{N}$ is a Carath\'eodory function satisfying the following assumptions:
	 \begin{itemize}
	 	\item[(i)] there exist constants $a_{1}>0$ and $a_{2},a_{3}\geq0$ satisfying
	 	\begin{equation*}
	 	\langle A(x,z,\xi),\xi\rangle \geq a_{1}\Phi(|\xi|)-a_{2}|z|^{\ell}-a_{3}, \quad x \in \Omega, z \in \mathbb{R}, \xi \in \mathbb{R}^{N};
	 	\end{equation*}
	 	\item[(ii)] There exists $a_{4} > 0$ such that $$|A(x,z,\xi)|\leq a_{4}\phi(|\xi|)|\xi|, \quad x \in \Omega, z \in \mathbb{R}, \xi \in \mathbb{R}^{N}.$$
	 \end{itemize}
    Here we also assume that 
    \begin{equation*}
    |g(x,t, z)|\leq C(|t|^{\alpha-1}+|t|^{r-1}), \quad \mbox{for all} \hspace{0,2cm} (x,t,z)\in \Omega\times\mathbb{R} \times \mathbb{R}^{N}
    \end{equation*}
    where $\alpha < r \leq \alpha^*$. In particular, assuming that conditions $(i)$ and $(ii)$ are satisfied with $\Phi(t)=t^{\ell}/\ell$, then we obtain the operator considered by P. Pucci and R. Servadei \cite{pucci2}. Moreover, our results complement \cite[Theorem 2]{cianchi}, since we have obtained regularity for weak solutions which are not necessarily local minimum of the associated energy functional.
\end{remark}

\begin{remark}\label{rem}
	We point out that Theorems~\ref{A}, \ref{B} and \ref{thrm} hold locally for any domain $\Omega\subset\mathbb{R}^{N}$. More precisely, let $f,a\in L^{q}_{loc}(\Omega)$ with $q\geq N/\alpha$ where $\alpha=\ell$ if $\Phi\not\approx t^{m}$ and $\alpha=m$ if $\Phi\approx t^{m}$. Assume also that $u\in W^{1,\Phi}_{loc}(\Omega)$ is a weak solution for the Problem \eqref{prob}. Then, we have the following conclusions:
	\begin{itemize}
		\item[(i)] If $q>N/\alpha$, then $u\in L^{\infty}_{loc}(\Omega)$. This fact follows by a slight adaptation of the proof of Theorem~\ref{A}, by considering the test function $\varphi:=|u|^{\beta_{k}/q'}\eta^{m}$, see \eqref{ml2}.
		\item[(ii)] If $q=N/\alpha$, then $u\in L^{q}_{loc}(\Omega)$, for all $q\in [1,+\infty)$. This fact follows by a slight adaptation of the proof of Theorem~\ref{B}, by considering the test function $\varphi:=u\min\{|u|^{\alpha s},L^{\alpha} \}\eta^{m}$, where $L$ is a positive parameter.
	\end{itemize}
	In the preceding items, $\eta\in C^{\infty}_{0}(\Omega)$ such that $0\leq\eta(x)\leq 1$, for all $x\in\Omega$.
\end{remark}

\begin{remark}
	Let $H(t):=\int_{0}^{t}h(s)\,\mathrm{d}s$ be a $N$-function which satisfies
	 \[
	  m<h^{-}:=\inf_{t>0}\frac{h(t)t}{H(t)}\leq \frac{h(t)t}{H(t)}\leq h^{+}:=\sup_{t>0}\frac{h(t)t}{H(t)}<\ell^{*}, \quad \mbox{for all} \hspace{0,2cm} t>0.
	 \]
	Moreover, following same ideas discussed in \cite{fang}, we also suppose that
	 \begin{equation}\label{tan}
	 	|g(x,t)|\leq C_{1}+C_{2}h(t), \quad  \mbox{for all} \hspace{0,2cm} (x,t)\in\Omega\times \mathbb{R}.
	 \end{equation}
	Notice that, in view of Lemma \ref{lema_naru}, one has
	 \[
	  \lim_{t\rightarrow \infty} \frac{h(t)}{t^{\ell^{*}-1}}\leq h(1)\lim_{t\rightarrow \infty}\frac{t^{h^{+}-1}}{t^{\ell^{*}-1}}=0.
	 \]
	For any $\varepsilon>0$ there exists $k>0$ such that if $|t|>k$, then $h(t)<\varepsilon t^{\ell^{*}-1}$. Consequently, there exists $C>0$ such that
	 \[
	  |g(x,t)|\leq C(1+|t|^{\ell^{*}-1}), \quad \mbox{for all} \hspace{0,2cm} (x,t)\in\Omega\times\mathbb{R}.
	 \]
	Thus, assuming that \eqref{tan} holds then we are able to apply Theorem~\ref{thrm} to conclude that any weak solution of Problem~\eqref{prob} belongs to $L^{\infty}(\Omega)$. Therefore, our main results complements \cite[Theorem 3.1]{fang} since we also consider the critical case and we required that only $\Phi$ satisfies the so called $\Delta_{2}$-condition. We mention that in light of \cite[Theorem 1.7]{lieberman} it follows that $u\in C^{1,\alpha}(\overline{\Omega})$, for some $\alpha\in(0,1)$.      
\end{remark}

\begin{remark}
	Notice that for our main results given in Theorems~\ref{B}, \ref{thrm} and Corollary \ref{cor1}, the Orlicz-Sobolev space $W^{1,\Phi}_{0}(\Omega)$ may not be a reflexive space, since we are also considering the extremal case $\ell=1$. More precisely, for the non-reflexive case, the conjugate function $\tilde{\Phi}$ does not satisfy $\Delta_{2}$-condition, see \cite{adams}.
\end{remark}

Notice that our main results complement some classical results by showing that any weak solution to the elliptic Problem \eqref{prob} is bounded. As was mentioned before, for quasilinear operators such as the $p$-Laplacian operator there exists several results concerning on regularity. On this subject we refer the reader to the important works \cite{dibene,peral,lad,pucci}. For this operator, choosing $\Phi(t) = |t|^{p}/p$ with $p > 1$, recall that Problem \eqref{probhomo} admits a bounded weak solution if and only if the nonlinearity $f$ is in $L^{q}(\Omega)$ for some $q > N/p$, see \cite{peral}. Furthermore, also for the $p$-Laplacian operator we know that any weak solution to the quasilinear elliptic problem \eqref{probhomo} is in $L^{q}(\Omega)$ for all $q \in [1, \infty)$ whenever $f$ is in $L^{N/p}(\Omega)$. Here we refer the reader to the important works \cite{pucci,pucci2}. For the $\Phi$-Laplacian operator there exists some preliminary results on regularity, see \cite{cianchi,fang}. However, to the best of our knowledge, there are not results on regularity taking into account the $\Phi$-Laplacian showing that weak solutions are bounded where the nonlinear term is critical. It is important to emphasize also that Theorem~\ref{B} jointly with Remark~\ref{rem} extend and complement \cite[Theorem 2.1]{pucci2}. Furthermore, in \cite{fang} the authors considered a more general class of nonlinearities with subcritical growth. In view of Remarks 1.8--1.11, the results obtained in \cite{fang} are extended in the present work since we deal with a more general operator together with subcritical and critical nonlinear term $g$ by showing that any weak solution to the elliptic problem \eqref{prob} remains bounded. 

The paper is organized as follows: Section \ref{s4} is devoted to the homogeneous case given in \eqref{probhomo} getting a proof for Theorem \ref{A}. In Section \ref{s3} we give some regularity results for the problem \eqref{prob} which provide us the proof of Theorem \ref{B}, Theorem \ref{thrm} and Corollary \ref{cor1},  Theorem \ref{smp}. In the Appendix we give an overview on Orlicz and Orlicz-Sobolev framework. Henceforth, we write $\int_{\Omega}f$ instead $\int_{\Omega}f(x)\;\mathrm{d}x$.


\section{The homogeneous case}\label{s4}

In order to obtain existence of solutions for \eqref{probhomo}, we introduce the following auxiliary problem
	\begin{equation}\label{aux}
	\left\{\
	\begin{array}{cl}
	\displaystyle-\Delta_\Phi u= f_{n}(x), & \mbox{in}~\Omega\\
	u=0, & \mbox{on}~\partial \Omega,
	\end{array}
	\right.
	\end{equation}
	where $f_{n}(x):=\min\{f(x),n\}$, $n\in\mathbb{N}$. The main idea is to get a sequence $\{u_{n}\}_{n} \subset W^{1,\Phi}_{0}(\Omega)$ which converges to a weak solution for the quasilinear elliptic problem \eqref{probhomo}. Moreover, such sequence has to be sufficiently regular in order to use $\varphi=u_{n}^{\beta_{k}/q'}$ as test function in \eqref{aux}, where $\beta_{k}$ will be defined later. In view of \cite[Lemma 3.1]{Fuk_2}, \cite[Theorem 1.1]{FG} and \cite[Corollary 3.1]{fang}, we conclude that for each $n\in\mathbb{N}$, Problem \eqref{aux} possesses an unique solution $u_{n}$ which belongs to $C^{1,\alpha_n}(\overline\Omega)$, for some $0<\alpha_n<1$. In light of the Comparison Principle \cite[Lemma 4.1]{fang}, the sequence of solutions $\{u_{n}\}_{n}$ is increasing, that is, using the fact that $f(x) \geq 0$  in $\Omega$ and $f \neq 0$, we obtain that 
	\[
	0<u_1\leq u_2\leq...\leq u_{n}\leq u_{n+1}\leq...
	\]
	Throughout this work we define $u^{-}(x) = - \max(u,0)$ for any $u \in \w$.
	
	From now on, for any $n\in\mathbb{N}$, we infer that the solution $u_{n}$ is positive. In fact, by using the negative part $-u_{n}^{-}$ as test function in \eqref{aux}, we can deduce that
	\[
	\ell\int_\Omega \Phi(|\nabla u_n^-|)\leq\int_\Omega \phi(|\nabla u^-_n|)|\nabla u_n^-|^2=-\int_\Omega f_nu_n^-\leq 0.
	\]
	Thus, $u_n^-\equiv 0$, that is, $u_n\geq 0$. Therefore, by using Strong Maximum Principle \cite[Theorem 1.1]{pucci} we conclude that $u_n>0$. Now we shall divide the proof of the existence result into three steps.
	
	\vspace{0,4cm}
	
	\noindent \textit{Step 1.} $\{u_{n}\}_{n}$ is a bounded sequence in $\w$.
	
	\vspace{0,4cm}
	
	In fact, since $q>N/\ell$ one has $q'<\ell^*$ ($q'<m^{*}$, if $q>N/m$) we have the continuous embedding $\w\hookrightarrow L^{q'}(\Omega)$. By using $u_n$ as test function in \eqref{aux} we obtain
	\begin{equation*}
	\ell\min\{\|u_n\|^\ell,\|u_n\|^m\}\leq\ell\int_\Omega\Phi(|\nabla u_n|)\leq\int_\Omega \phi(|\nabla u_n|)|\nabla u_n|^2
	\leq C\|f\|_q\|u_n\|,
	\end{equation*}
	which implies that $\{u_n\}_{n}$ is bounded in $\w$. As a consequence, we know that $u_n\rightharpoonup u$ weakly in $\w$.
	
	\vspace{0,4cm}
	
	\noindent \textit{Step 2.} $u_n\rightarrow u$ strongly in $\w$.
	
	\vspace{0,4cm}
	
   By taking $u_n-u$ as test function in \eqref{aux} and using the compact embedding $\w\hookrightarrow L^{q'}(\Omega)$ we obtain
	\begin{equation*}
	\langle -\Delta_\Phi u_n,u_n-u\rangle = \int_\Omega f_n(u_n-u)
	\leq \int_\Omega f|u_n-u|
	\leq  \|f\|_q\|u_n-u\|_{q'}\rightarrow 0.
	\end{equation*}
	Therefore, in view of condition $(S_+)$ (see \cite[Theorem~4]{le}) we conclude that $u_n\rightarrow u$ strongly in $\w$.
	
	\vspace{0,4cm}
	
	\noindent \textit{Step 3.} The function $u$ described above is a weak solution for the homogeneous quasilinear elliptic problem ~\eqref{probhomo}.
	
	\vspace{0,4cm}
	
	According to Step 2 we observe that $\nabla u_n\rightarrow \nabla u$ a.e. in $\Omega$, see \cite{murat}. It follows that
	\[
	\phi(|\nabla u_n|)\nabla u_n\rightarrow \phi(|\nabla u|)\nabla u, \quad \mbox{a.e. in } \Omega.
	\]
	Since $\{\phi(|\nabla u_n|)\nabla u_n\}_{n}$ is bounded in $\prod L_{\widetilde{\Phi}}(\Omega)$, it follows from \cite[Lemma 2]{gossez-Czech} that
	\[
	\int_\Omega\phi(|\nabla u_n|)\nabla u_n\nabla v \rightarrow \int_\Omega\phi(|\nabla u|)\nabla u\nabla v , \quad \mbox{for all} \hspace{0,2cm} v\in\w.
	\]
	On the other hand, since $f_n\leq f$ and $f_nv\rightarrow fv$ a.e. in $\Omega$, by using Lebesgue Dominated Convergence Theorem we conclude that
	\[
	\int_{\Omega} f_n v \rightarrow \int_{\Omega} fv, \quad \mbox{for all} \hspace{0,2cm} v\in \w.
	\]
	The last identity implies that $u$ is a weak solution for the quasilinear elliptic problem \eqref{probhomo}. Now, we are concerned with the regularity for the Problem \eqref{probhomo}.

\begin{proof}[{\bf Proof of Theorem~\ref{A}~(i)}] The main idea is to apply a Moser iteration method. Let us introduce the following sequence
 \[
  \beta_1=q'(\ell-1), \quad \beta_k^*=\beta_k+\beta_1 \quad \mbox{and} \quad \beta_{k+1}=\delta\beta_k^*,
 \]
where $\delta:=\ell^*/(\ell q')$. Note that since $q>N/\ell$ which implies that $\delta>1$. Thus, we can deduce that
 \begin{equation}\label{ml1}
  \beta_k^*=\left(2\delta^{k-1}+\delta^{k-2}+...+1\right)\beta_1=\frac{2\delta^k-\delta^{k-1}-1}{\delta-1}\beta_1,
 \end{equation}
 \begin{equation}\label{ml2}
  \beta_k=\left(2\delta^{k-1}+\delta^{k-2}+...+\delta\right)\beta_1=\frac{2\delta^k-\delta^{k-1}-\delta}{\delta-1}\beta_1.
 \end{equation}
Since $\{\beta_{k}\}_{k}$ is a increasing sequence and $\beta_k\rightarrow+\infty$, as $k\rightarrow+\infty$, let us consider $k_0\in\mathbb{N}$ be such that $\beta_{k_0}-q'\geq q'$, for all $k\geq k_{0}$. By taking $\varphi=u_n^{\frac{\beta_k}{q'}}$ as test function in \eqref{aux} and using H\"{o}lder inequality we get
 \begin{equation}\label{equ1}
  \frac{\beta_k}{q'}\int_\Omega\phi(|\nabla u_n|)|\nabla u_n|^2u_n^{\frac{\beta_k}{q'}-1}\leq \int_{\Omega} f_nu_n^{\frac{\beta_k}{q'}}\leq \|f\|_q\|u_n\|_{\beta_k}^{\frac{\beta_k}{q'}}.
 \end{equation}	
On the other hand, in view of Proposition \ref{lema_naru}, one has
 \begin{equation}\label{eq_2}
  \frac{\beta_k}{q'}\int_\Omega\phi(|\nabla u_n|)|\nabla u_n|^2u_n^{\frac{\beta_k}{q'}-1} \geq \frac{\ell\Phi(1)}{q'}\beta_k\int_{\{|\nabla u_n|\geq 1\}}|\nabla u_n|^\ell u_n^{\frac{\beta_k}{q'}-1}.
 \end{equation}
Combining \eqref{equ1} and \eqref{eq_2} we obtain
 \begin{equation}\label{equ3}
  \frac{\ell\Phi(1)}{q'}\beta_k\int_{\Omega}|\nabla u_n|^\ell u_n^{\frac{\beta_k}{q'}-1}\leq \|f\|_q\|u_n\|_{\beta_k}^{\frac{\beta_k}{q'}}+\frac{\ell\Phi(1)}{q'}\beta_k\int_{\Omega} u_n^{\frac{\beta_k}{q'}-1}.
 \end{equation}
By using the embedding $L^{{\beta_k}}(\Omega)\hookrightarrow L^{\frac{\beta_k}{q'}-1}(\Omega)$ we also get
 \begin{eqnarray}\label{eq-u_1}
  \int_\Omega u_n^{\frac{\beta_k}{q'}-1}\leq |\Omega|^{1-\frac{1}{q'}+\frac{1}{\beta_k}}\|u_n\|_{\beta_k}^{\frac{\beta_k}{q'}}\|u_n\|_{\beta_k}^{-1}.
 \end{eqnarray}
In view of the embedding $L^{{\beta_k}}(\Omega)\hookrightarrow L^1(\Omega)$ we infer that
 \begin{equation*}
  \|u_1\|_1\leq |\Omega|^{1-\frac{1}{\beta_k}}\|u_1\|_{\beta_k}.
 \end{equation*}
Since $u_1\leq u_n$, it follows that $\|u_1\|_{\beta_k}\leq \|u_n\|_{\beta_k}$. It is no hard to verify that 
 \begin{equation}\label{eq-u_1-4}
  \|u_n\|_{\beta_k}^{-1}\leq |\Omega|^{1-\frac{1}{\beta_k}}\|u_1\|_1^{-1}.
 \end{equation}
Thus, combining \eqref{eq-u_1} and \eqref{eq-u_1-4} we conclude that
 \begin{eqnarray}\label{eq-u_1-5}
  \int_\Omega u_n^{\frac{\beta_k}{q'}-1}\leq |\Omega|^{2-\frac{1}{q'}}\|u_1\|_{1}^{-1}\|u_n\|_{\beta_k}^{\frac{\beta_k}{q'}}.
 \end{eqnarray}
By using \eqref{equ3} and \eqref{eq-u_1-5} we obtain
 \begin{eqnarray}
  \frac{\ell\Phi(1)}{q'}\beta_k\int_{\Omega}|\nabla u_n|^\ell u_n^{\frac{\beta_k}{q'}-1} &\leq& \left(\|f\|_q+\frac{\ell\Phi(1)}{q'}\beta_k|\Omega|^{2-\frac{1}{q'}}\|u_1\|_{1}^{-1}\right)\|u_n\|_{\beta_k}^{\frac{\beta_k}{q'}}\nonumber\\
  &\leq& \beta_k\left(\|f\|_q+\frac{\ell\Phi(1)}{q'}|\Omega|^{2-\frac{1}{q'}}\|u_1\|_{1}^{-1}\right)\|u_n\|_{\beta_k}^{\frac{\beta_k}{q'}}\nonumber.
 \end{eqnarray}
Thus, we have concluded that
 \begin{equation}\label{ult-1}
  \int_{\Omega}|\nabla u_n|^\ell u_n^{\frac{\beta_k}{q'}-1} \leq A\|u_n\|_{\beta_k}^{\frac{\beta_k}{q'}},
 \end{equation}
where
 \[
  A:=\frac{q'}{\ell\Phi(1)}\left(\|f\|_q+\frac{\ell\Phi(1)}{q'}|\Omega|^{2-\frac{1}{q'}}\|u_1\|_{1}^{-1}\right).
 \]
Notice that
 \begin{eqnarray}\label{eq_6}
  \displaystyle\left(\frac{\ell q'}{\beta_k+\beta_1}\right)^\ell\int_\Omega\left|\nabla \left(u_n^{\frac{\beta_k+\beta_1}{\ell q'}}\right)\right|^\ell=\int_\Omega|\nabla u_n|^\ell u_n^{\frac{\beta_k}{q'}-1}.
 \end{eqnarray}
Combining \eqref{ult-1} and \eqref{eq_6} we deduce that
 \begin{equation*}
  \int_\Omega\left|\nabla \left(u^{\frac{\beta_k^*}{\ell q'}}\right)\right|^\ell  \leq A\left(\frac{\beta_k^*}{\ell q'}\right)^\ell\|u_n\|_{\beta_k}^{\frac{\beta_k}{q'}}.
 \end{equation*}
In view of the embedding $W_0^{1,\ell}(\Omega)\hookrightarrow L^{\ell^*}(\Omega)$, there exists $\mu>0$ such that
 \begin{equation}\label{eq_8}
  \|u_n\|_{\beta_{k+1}}^{\frac{\beta_k^*}{q'}}=\left\|u^{\frac{\beta_k^*}{\ell q'}}\right\|_{\ell^*}^\ell  \leq \mu^\ell A\left(\frac{\beta_k^*}{\ell q'}\right)^\ell\|u_n\|_{\beta_k}^{\frac{\beta_k}{q'}}.
 \end{equation}
Let us define $F_{k+1}:=\beta_{k+1}\ln \|u_n\|_{\beta_{k+1}}$. It follows from \eqref{eq_8} that
 \begin{eqnarray}
 F_{k+1} &\leq & \frac{\beta_{k+1}q'}{\beta_k^*}\left(\ell\ln \mu +\ell\ln\left(\frac{\beta_k^*}{\ell q'}\right)+\ln A+\frac{\beta_k}{q'}\ln\|u_n\|_{\beta_k}\right)\nonumber\\
         &\leq& \ell^*\ln\left(\mu A\beta_k^*\right)+ \frac{\ell^*}{q'\ell}F_k\nonumber\\
         &=&\lambda_k+\delta F_k\nonumber,
 \end{eqnarray}
where $\lambda_k:=\ell^*\ln\left(\mu A\beta_k^*\right)$. By using \eqref{ml1} and \eqref{ml2} we deduce that
 \[
  \lambda_k=b+\ell^*\ln\left(2\delta^{k-1}+\delta^{k-2}+...+1\right),
 \]
where $b:=\ell^*\ln(\mu A\beta_1)$. Hence, we get
 \begin{equation}\label{eq_10}
  \frac{\lambda_n}{\delta^n} = \frac{b}{\delta^n}+\frac{\ell^*}{\delta^n}\ln\left(\frac{2\delta^n-\delta^{n-1}-1}{\delta-1}\right)
  \leq \frac{b}{\delta^n}+\frac{\ell^*}{\delta^n}\ln\left(\frac{2\delta^n}{\delta-1}\right).
\end{equation}
Furthermore, we also mention that 
 \[
  F_k\leq \delta^{k-1}F_1+\lambda_{k-1}+\delta\lambda_{k-2}+...+\delta^{k-2}\lambda_1.
 \]
This inequality shows that 
 \begin{eqnarray}\label{eq_9}
  \frac{F_k}{\beta_k} & \leq & \frac{\displaystyle F_1+\frac{\lambda_{k-1}}{\delta^{k-1}}+\frac{\lambda_{k-2}}{\delta^{k-2}}+...+\frac{\lambda_1}{\delta}}{\displaystyle\frac{2\delta-1-\displaystyle\frac{1}{\delta^{k-1}}}{\delta-1}\beta_1}.
 \end{eqnarray}
Combining \eqref{eq_9} and \eqref{eq_10} we obtain
 \begin{eqnarray}
  \frac{F_k}{\beta_k} & \leq & \frac{F_1+b\left(\displaystyle\frac{1}{\delta^{k-1}}+...\frac{1}{\delta}\right)+\ell^*\left(\displaystyle\frac{1}{\delta^{k-1}}\ln\left(\frac{2\delta^{k-1}}{\delta-1}\right)+...+
	\frac{1}{\delta}\ln\left(\frac{2\delta}{\delta-1}\right)\right)}{\displaystyle\frac{2\delta-1-1/\delta^{k-1}}{\delta-1}\beta_1}\nonumber\\
  & \leq & \frac{F_1+\displaystyle\frac{b}{\delta-1}+\ell^*\left[\left(\frac{1}{\delta^{k-1}}+...+\frac{1}{\delta}\right)\ln\left(\frac{2}{\delta-1}\right)+\left(\frac{k-1}{\delta^{k-1}}+...+\frac{1}{\delta}\right)\ln(\delta)\right]}{\frac{\displaystyle 2\delta-1-1/\delta^{k-1}}{\displaystyle\delta-1}\beta_1}\nonumber\\
  & \leq & \frac{F_1+\displaystyle\frac{b}{\delta-1}+\ell^*\left[\frac{1}{\delta-1}\ln\left(\frac{2}{\delta-1}\right)+\ln(\delta)\sum_{n=1}^\infty \frac{n}{\delta^n}\right]}{\displaystyle\frac{2\delta-1-1/\delta^{k-1}}{\delta-1}\beta_1}\longrightarrow d_0\nonumber.
 \end{eqnarray}
As a consequence, using  the definition of $F_k$ we also conclude that
 \[
  \|u_n\|_\infty=\limsup_{k\rightarrow\infty}\|u_n\|_{\beta_k}\leq \limsup_{k\rightarrow\infty} e^{\frac{F_k}{\beta_k}}\leq e^{d_0}.
 \]
At this stage, we also mention that
 \[
  \|u\|_\infty\leq \liminf_{n\rightarrow\infty}\|u_n\|_\infty\leq e^{d_0},
 \]
which implies that $u\in L^{\infty}(\Omega)$. This ends the proof. 
\end{proof}

\begin{proof}[{\bf Proof of Theorem~\ref{A}~(ii)}]
The proof of Theorem~\ref{A}~$(ii)$ follows by similar arguments from the proof of Theorem~\ref{A}~$(i)$. Let $\{u_{n}\}_{n}$ be the sequence of solutions for \eqref{aux}. Under this condition, using the fact that $\Phi\approx t^{m}$, there exist $C,T>0$ such that
 \[
   Ct^{m}\leq\Phi(t), \quad \mbox{for all} \hspace{0,2cm} t\geq T.
 \]
Moreover, since $q'<m^{*}$ it follows that
 \[
  W^{1,\Phi}_{0}(\Omega)=W^{1,m}_{0}(\Omega)\hookrightarrow L^{m^{*}}(\Omega)\hookrightarrow L^{q'}(\Omega).
 \]
Arguing as in \textit{Step} 1 we infer that $\{u_{n}\}_{n}$ is a bounded sequence in $W^{1,\Phi}_{0}(\Omega)$. Following the same ideas discussed in the proof of Theorem~\ref{A}~$(i)$, we define $\beta_{1}=q'(m-1)$ and $\delta=m^{*}/(mq')$. Now, we also change \eqref{eq_2} by
 \[
    \frac{\beta_k}{q'}\int_\Omega\phi(|\nabla u_n|)|\nabla u_n|^2u_n^{\frac{\beta_k}{q'}-1} \geq \frac{\ell\Phi(1)}{q'}\beta_k\int_{\{|\nabla u_n|\geq T\}}|\nabla u_n|^m u_n^{\frac{\beta_k}{q'}-1}.
 \]
Moreover, the estimate \eqref{equ3} can be rewritten in the following form
 \[
  \frac{\ell\Phi(1)}{q'}\beta_k\int_{\Omega}|\nabla u_n|^m u_n^{\frac{\beta_k}{q'}-1}\leq \|f\|_q\|u_n\|_{\beta_k}^{\frac{\beta_k}{q'}}+\frac{\ell T^{m}\Phi(1)}{q'}\beta_k\int_{\Omega} u_n^{\frac{\beta_k}{q'}-1}.
 \]
In order to deduce \eqref{eq_8}, we use the embedding $W^{1,\Phi}_{0}(\Omega)\hookrightarrow L^{m^{*}}(\Omega)$. Henceforth, the proof follows analogously to the proof of Theorem~\ref{A}~$(i)$. We omit the details. 
\end{proof}

\section{The nonhomogeneous case}\label{s3}

In this Section we consider the nonhomogeneous problem given by \eqref{prob}. In order to obtain regularity, we shall use a Moser's iteration method, see \cite{pucci2,struwe}.
Before starting the procedure, we consider a useful estimate which will be crucial in the method.

\begin{lemma}\label{est}
	Let $u\in W^{1,\Phi}_{0}(\Omega)$ be a weak solution of \eqref{prob} and $s,L$ positive parameters. Then, it holds
	 \begin{equation}\label{emj1}
	  |\nabla (u\min\{|u|^{s},L \})|^{\ell}\leq c_{\ell}\left\{|\nabla u|^{\ell}\min\{|u|^{\ell s},L^{\ell} \}+(2s+s^{2})^{\ell/2}|\nabla u|^{\ell}|u|^{\ell s}\mathcal{X}_{\{|u|^{s}\leq s \}} \right\},
	 \end{equation}
	where $\mathcal{X}_{\{|u|^{s}\leq L\}}$ denotes the characteristic function over the set $\{|u|^{s}\leq L\}$ and
	 \[
	  c_{\ell}=\left\{
	   \begin{array}{cl}
	    1, & \ell\leq 2,\\
	    2^{\ell/2-1}, & \ell >2.
	   \end{array}
	  \right.
	 \]
\end{lemma}
\begin{proof}
	A simple computation leads to
	 \begin{equation*}
	  |\nabla (u\min\{|u|^{s},L \})|^{2}=|\nabla u|^{2}\min\{|u|^{2s},L^{2} \}+(2s+s^{2})|\nabla u|^{2}|u|^{2s}\mathcal{X}_{\{|u|^{s}\leq L \}}.
	 \end{equation*}
	Now we divide into two cases. Namely, we consider the cases $\ell\leq 2$ and $\ell>2$. If $\ell\leq 2$, then the function $t\mapsto t^{\ell/2}$ is concave. Thus, one can deduce
	 \begin{equation}\label{emj2}
	  |\nabla (u\min\{|u|^{s},L \})|^{\ell}\leq |\nabla u|^{\ell}\min\{|u|^{\ell s},L^{\ell} \}+(2s+s^{2})^{\ell/2}|\nabla u|^{\ell}|u|^{\ell s}\mathcal{X}_{\{|u|^{s}\leq L \}}.
	 \end{equation}
	If $\ell>2$, then the function $t\mapsto t^{\ell/2}$ is convex. Thus, we obtain
	 \begin{equation}\label{emj3}
	 	|\nabla (u\min\{|u|^{s},L \})|^{\ell}\leq 2^{\ell/2-1}\left\{|\nabla u|^{\ell}\min\{|u|^{\ell s},L^{\ell} \}+(2s+s^{2})^{\ell/2}|\nabla u|^{\ell}|u|^{\ell s}\mathcal{X}_{\{|u|^{s}\leq L \}}\right\}.
	 \end{equation}
	Combining \eqref{emj2} and \eqref{emj3} we get \eqref{emj1}. This finishes the proof. 
\end{proof}

\begin{proof}[{\bf Proof of Theorem~\ref{B}}] Here we shall prove the item $(i)$. The proof for the  case $(ii)$ is a direct adaptation of the proof of case $(i)$ together with a similar procedure of the proof of Theorem~\ref{A}~$(ii)$. Let $u\in W^{1,\Phi}_{0}(\Omega)$ be a weak solution of \eqref{prob} and $L>0$. Notice that
 \[
  \nabla (u\min\{|u|^{\ell s},L^{\ell} \})=\min\{|u|^{\ell s},L^{\ell} \}\nabla u+\ell s|u|^{\ell s}\mathcal{X}_{\{|u|^{s}\leq L\}}\nabla u.
 \]
Thus, by taking $\varphi:=u\min\{|u|^{\ell s},L^{\ell} \}$ as test function in \eqref{prob}, one can deduce that
 \begin{align}\label{emj4}
 	\int\phi(|\nabla u|)|\nabla u|^{2}\min\{|u|^{\ell s},L^{\ell} \}+\ell s\int_{\{|u|^{s}\leq L \}}\phi(|\nabla u|)|\nabla u|^{2}|u|^{\ell s}\leq \int\left[a(x)(1+2|u|^{\ell}\min\{|u|^{\ell s},L^{\ell} \})\right],
 \end{align}
for $L>0$ sufficiently large. Let us define $h:(0,+\infty)\rightarrow\mathbb{R}$ given by
 \begin{equation}\label{g}
  h(s):=c_{\ell}\max\left\{1,\frac{(2s+s^{2})^{\ell/2}}{\ell s}\right\}.
 \end{equation}
Notice that $h(s)\leq c_{\ell}(1+s)^{\ell}$. In view of Lemma~\ref{est}, Proposition \ref{lema_naru}, estimate \eqref{emj4} and H\"{o}lder inequality one has
 \begin{align*}
  \int_{\{|\nabla u|\geq 1 \}}|\nabla(u\min\{|u|^{s},L \})|^{\ell} & \leq c_{\ell}(1+s)^{\ell}\left\{\int_{\{|\nabla u|\geq 1 \}}|\nabla u|^{\ell}\min\{|u|^{\ell s},L^{\ell} \}+\ell s\int_{\{|\nabla u|\geq 1 \}\cap\{|u|^{s}\leq L \}}|\nabla u|^{\ell}|u|^{\ell s}\right\}\\
                                                                           & \leq \frac{c_{\ell}(1+s)^{\ell}}{\phi(1)}\left\{\int\phi(|\nabla u|)|\nabla u|^{2}\min\{|u|^{\ell s},L^{\ell} \}+\ell s\int_{\{|u|^{s}\leq L \}}\phi(|\nabla u|)|\nabla u|^{2}|u|^{\ell s}\right\}\\
                                                                           & \leq \frac{c_{\ell}(1+s)^{\ell}}{\phi(1)}\int a(x)(1+2|u|^{\ell}\min\{|u|^{\ell s},L^{\ell} \})\\
                                                                           & \leq \frac{c_{\ell}(1+s)^{\ell}}{\phi(1)}\left(\|a\|_{1}+2k\|u\|_{\ell(s+1)}^{\ell(s+1)}\right)+2\frac{c_{\ell}(1+s)^{\ell}}{\phi(1)}\left(\int_{\{|a|\geq k \}}|a(x)|^{N/\ell} \right)^{\ell/N}\|u \min\{|u|^{s},L \}\|_{\ell^{*}}^{\ell},
 \end{align*}
where $k=k(s)>0$ is a parameter which depends on $s$. Let $S>0$ be the sharp constant of the continuous embedding $W^{1,\ell}_{0}(\Omega)\hookrightarrow L^{\ell^{*}}(\Omega)$. Taking into account the above estimates we obtain that 
 \[
  \int_{\{|\nabla u|\geq 1 \}}|\nabla(u\min\{|u|^{s},L \})|^{\ell}\leq \frac{c_{\ell}(1+s)^{\ell}}{\phi(1)}\left(\|a\|_{1}+2k\|u\|_{\ell(s+1)}^{\ell(s+1)}\right)+2S\frac{c_{\ell}(1+s)^{\ell}}{\phi(1)}\left(\int_{\{|a|\geq k \}}|a(x)|^{N/\ell} \right)^{\ell/N}\|u \min\{|u|^{s},L \}\|^{\ell}.
 \]
Since $a\in L^{N/\ell}(\Omega)$, for given $s>0$ there exists $k=k(s)>0$ such that
 \[
  2S\frac{c_{\ell}(1+s)^{\ell}}{\phi(1)}\left(\int_{\{|a|\geq k \}}|a(x)|^{N/\ell} \right)^{\ell/N}=\frac{1}{2}.
 \]
Hence, we obtain
 \begin{equation}\label{emj5}
  \frac{1}{2}\int|\nabla(u\min\{|u|^{s},L \})|^{\ell}\leq \frac{c_{\ell}(1+s)^{\ell}}{\phi(1)}\left(\|a\|_{1}+2k\|u\|_{\ell(s+1)}^{\ell(s+1)}\right)+\int_{\{|\nabla u|\leq 1 \}}|\nabla (u \min\{|u|^{s},L \})|^{\ell}.
 \end{equation} 
By using Lemma~\ref{est} we deduce
 \begin{align}\label{emj6}
  \int_{\{|\nabla u|\leq 1 \}}|\nabla (u \min\{|u|^{s},L \})|^{\ell} & \leq c_{\ell}\left\{\int_{\{|\nabla u|\leq 1 \}}\min\{|u|^{\ell s},L^{\ell} \}+(2s+s^{2})^{\ell/2}\int_{\{|\nabla u|\leq 1 \}\cap\{|u|^{s}\leq 1 \}}|u|^{\ell s} \right\}\nonumber\\
                                                                     & \leq c_{\ell}\left\{1+(2s+s^{2})^{\ell/2} \right\}\|u\|_{\ell s}^{\ell s}\nonumber\\
                                                                     & \leq c_{\ell}\left\{1+(2s+s^{2})^{\ell/2} \right\}|\Omega|^{1-\frac{\ell s}{\ell(s+1)}}\|u\|_{\ell(s+1)}^{\ell s}.
 \end{align}
Combining \eqref{emj5}, \eqref{emj6} and taking the limit $L\rightarrow+\infty$ we conclude that
 \begin{equation*}
  \|\nabla (|u|^{s}u)\|_{\ell}^{\ell}\leq \tilde{c}_{\ell}(a,\Omega)(1+s)^{\ell}\left\{1+k\|u\|_{\ell(s+1)}^{\ell(s+1)}+\|u\|_{\ell(s+1)}^{\ell s} \right\}.
 \end{equation*}
Now, using the embedding $W_0^{1,\ell}(\Omega)\hookrightarrow L^{\ell^*}(\Omega)$ we get
 \begin{equation}\label{emj7}
 	\|u\|_{\ell^{*}(s+1)}^{\ell(s+1)}\leq \tilde{c}_{\ell}(a,\Omega)(1+s)^{\ell}\left\{1+k\|u\|_{\ell(s+1)}^{\ell(s+1)}+\|u\|_{\ell(s+1)}^{\ell s} \right\}.
 \end{equation}
In light of the general estimate \eqref{emj7}, we are able to start the iteration procedure considering
 \[
  s_{0}=0 \quad \mbox{and} \quad s_{i}+1=(s_{i-1}+1)\frac{N}{N-\ell}, \quad i=1,2,...
 \]
Therefore, for each $q\in[1,+\infty)$, there exists $i\in\mathbb{N}$ such that 
 \[
  \ell\left(\frac{N}{N-\ell} \right)^{i}>q \quad \mbox{and} \quad u\in L^{\ell\left(\frac{N}{N-\ell} \right)^{i}}(\Omega).
 \]
This estimate finishes the proof of Theorem~\ref{B}.
\end{proof}

\begin{proof}[{\bf Proof of Corollary \ref{cor1}}]
	Notice that
	 \[
	  |g(x,u)|\leq a(x)(1+|u|^{\alpha-1}), \quad \mbox{where} \hspace{0,2cm} a(x):=\frac{C(|u|^{\alpha-1}+|u|^{\alpha^{*}-1})}{1+|u|^{\alpha-1}}\in L^{N/\alpha}(\Omega),
	 \]
	where $\alpha\in\{\ell,m \}$. Therefore, the desired result follows immediately from Theorem~\ref{B}.
\end{proof}

\begin{proof}[{\bf Proof of Theorem \ref{thrm}}]
   Now we shall prove the case where \eqref{ell} holds true. In this case, assuming also that $\alpha=m$, the proof follows by slight modifications as in the previous results. In light of Corollary \ref{cor1}, we have that $u\in L^{q}(\Omega)$, for all $q\in[1,\infty)$. Let $h:(0,+\infty)\rightarrow\mathbb{R}$ be the function defined in \eqref{g}. Let us define $\varphi:=u\min\{|u|^{\ell s},L^{\ell} \}$. For the reader convenience, we introduce the notation $\psi:=u\min\{|u|^{s},L \}$. By using $\varphi$ as test function in \eqref{prob} and similar calculations to the proof of Theorem~\ref{B} we deduce also that 
    \[
     \int_{\{|\nabla u|\geq 1 \}} |\nabla\psi|^{\ell}\leq \tilde{c}_{\ell}(1+s)^{\ell}\int |\psi|^{\ell}(1+|u|^{\ell^{*}-\ell}).
    \]
   Thus, we deduce that
    \[
     \|\nabla\psi\|_{\ell}^{\ell} \leq \tilde{c}_{\ell}(1+s)^{\ell}\int |\psi|^{\ell}(1+|u|^{\ell^{*}-\ell})+c_{\ell}|\Omega|^{1-\frac{\ell s}{\ell(s+1)}}[1+(2s+s^{2})^{\ell/2}]\|u\|_{\ell(s+1)}^{\ell s},
    \]
   which implies
    \begin{equation*}
    	\|\nabla\psi\|_{\ell}^{\ell} \leq \bar{c}_{\ell}(\Omega)(1+s)^{\ell}\left\{\int |\psi|^{\ell}(1+|u|^{\ell^{*}-\ell})+\|u\|_{\ell(s+1)}^{\ell s}\right\},
    \end{equation*}
 where $\bar{c}_{\ell}(\Omega):=\max\{\tilde{c}_{\ell},c_{\ell}\sup_{s\geq0}|\Omega|^{1-\frac{\ell s}{\ell(s+1)}} \}$. Let $r:=[(\ell^{*})^{2}-\ell\ell^{*}+\ell^{2}]/(\ell\ell^{*})>1$ be a fixed number. By using H\"{o}lder inequality with $\ell^{*}/(\ell r)$ and $(\ell^{2})^{2}/\ell(\ell^{*}-\ell)$ we obtain
  \[
   S\|\nabla\psi\|_{\ell^{*}}^{\ell} \leq \bar{c}_{\ell}(\Omega)(1+s)^{\ell}\left\{\|\psi\|_{\ell^{*}/r}^{\ell}\|u\|_{(\ell^{*})^{2}/\ell}+\|\psi\|_{\ell}^{\ell}+\|u\|_{\ell(s+1)}^{\ell s}\right\},
  \]
 where $S>0$ is the sharp constant of the continuous embedding $W^{1,\ell}_{0}(\Omega)\hookrightarrow L^{\ell^{*}}(\Omega)$. By using Lebesgue Dominated Convergence Theorem taking $L\rightarrow+\infty$ we get
  \begin{equation}\label{emj8}
   \|u\|_{\ell^{*}(s+1)}^{\ell(s+1)} \leq \bar{c}_{\ell}(\Omega)(1+s)^{\ell}\left\{\|u\|_{\frac{\ell^{*}}{r}(s+1)}^{\ell(s+1)}\|u\|_{(\ell^{*})^{2}/\ell}+\|\psi\|_{\ell}^{\ell}+\|u\|_{\ell(s+1)}^{\ell s}\right\}
  \end{equation}
 In view of the continuous embedding $L^{\frac{\ell^{*}}{r}(s+1)}(\Omega)\hookrightarrow L^{\ell(s+1)}(\Omega)$ we deduce the following estimates
  \begin{equation}\label{emj9}
  	\|u\|_{\ell(s+1)}^{\ell(s+1)}\leq |\Omega|^{1-\frac{\ell r}{\ell^{*}}}\|u\|_{\frac{\ell^{*}}{r}(s+1)}^{\ell(s+1)},
  \end{equation}
  and
  \begin{equation}\label{emj10}
  	\|u\|_{\ell(s+1)}^{\ell s}\leq |\Omega|^{\left(\frac{1}{\ell}-\frac{r}{\ell^{*}} \right)\frac{\ell s}{s+1}}\|u\|_{\frac{\ell^{*}}{r}(s+1)}^{\ell s}.
  \end{equation}
 Combining \eqref{emj8}, \eqref{emj9} and \eqref{emj10} we conclude that
  \[
   \|u\|_{\ell^{*}(s+1)}\leq k^{1/(s+1)}(1+s)^{\ell/(s+1)}\max\left\{\|u\|_{\frac{\ell^{*}}{r}(s+1)},1 \right\}, \quad \mbox{for all} \hspace{0,2cm} s\in(0,+\infty).
  \]
 At this stage, choosing $s+1=r$, one has
  \[
   \|u\|_{\ell^{*}r}\leq k^{1/r}r^{\ell/r}\max\left\{\|u\|_{\ell^{*}},1 \right\}.
  \]
 Now, we continue the iteration by taking $s+1=r^{2}$. Thus, we obtain the following estimate 
  \[
   \|u\|_{\ell^{*}r^{2}}\leq k^{1/r^{2}}r^{\ell/r^{2}}\max\left\{k^{1/r}r^{\ell/r}\max\left\{ \|u\|_{\ell^{*}},1\right\},1 \right\}.
  \]
 By iterating similarly to \cite[p. 3344]{pucci2}, we conclude that $u\in L^{\infty}(\Omega)$.
 This ends the proof.  
\end{proof}

\begin{proof}[{\bf Proof of Theorem \ref{smp}}]
	In view of Theorem~\ref{thrm} it follows that $u\in L^{\infty}(\Omega)$. Hence, by using \cite[Theorem 1.7]{lieberman} we conclude that $u\in C^{1,\alpha}(\overline{\Omega})$, for some $\alpha\in(0,1)$. Let us define $H(t):=t^{2}\phi(t)-\Phi(t)$. It is not hard to check that $H$ is an increasing function and satisfies $\Phi^{-1}(s)\geq H^{-1}((\ell-1)s)$ for any $s \geq 0$. Taking into account \eqref{ell} we deduce that
	 \[
	  H^{-1}(G(x,t))\leq C\Phi^{-1}(|t|^{\ell}+|t|^{\ell^{*}}),  x \in \Omega, t \in \mathbb{R}.
	 \] 
	The last assertion implies that
	 \[
	  \int_{0}^{\delta}\frac{\mathrm{d}s}{H^{-1}(G(x,s))}=\infty
	 \]
holds true for some $\delta > 0$. Therefore, we are able to use the Strong Maximum Principle given in \cite[Theorem 1.1]{pucci} showing that $u>0$ in $\Omega$. This ends the proof. 
\end{proof}

\vspace{0,5cm}

\section*{Appendix}

In this appendix we recall some basic concepts on Orlicz and Orlicz-Sobolev spaces. For a more complete discussion on this subject we refer the readers to \cite{adams,rao}. Let $\Theta:\mathbb{R}\rightarrow[0,+\infty)$ be convex and continuous. It is important to say that $\Theta$ is a $N$-function if $\Theta$ satisfies the following conditions:
\begin{itemize}
	\item[(i)] $\Theta$ is even;
	\item[(ii)] $\displaystyle\lim_{t\rightarrow 0}\frac{\Theta(t)}{t}=0$;
	\item[(iii)] $\displaystyle\lim_{t\rightarrow \infty}\frac{\Theta(t)}{t}=\infty$;
	\item[(iv)] $\Theta(t)>0$, for all $t>0$.	
\end{itemize}
Notice that by using assumptions $(\phi_1)$ and $(\phi_2)$ we conclude that $\Phi$, defined in \eqref{phi}, is a $N$-function. Henceforth, $\Phi$ and $\Psi$ denote $N$-functions. 

Recall also that a $N$-function satisfies the $\Delta_{2}$-condition if there exists $K>0$ such that
\[
\Phi(2t)\leq K\Phi(t), \quad \mbox{for all} \hspace{0,2cm} t\geq0.
\]
We denote by $\tilde{\Phi}$ the complementary function of $\Phi$, which is given by the Legendre's transformation
\[
\tilde{\Phi}(s)=\max_{t\geq0}\{st-\Phi(t)\}, \quad \mbox{for all} \hspace{0,2cm} s\geq0.
\]
Let $\Omega\subset\mathbb{R}^{N}$ be an open subset and $\Phi:[0,+\infty)\rightarrow[0,+\infty)$ be fixed. The set
\[
\mathcal{L}_{\Phi}(\Omega):=\left\{u:\Omega\rightarrow\mathbb{R} \mbox{ measurable } : \int_{\Omega}\Phi(|u(x)|)<+\infty\right\},
\]
is the so-called \textit{Orlicz class}. Let us suppose that $\Phi$ is a Young function generated by $\varphi$, that is
\[
\Phi(t)=\int_{0}^{t}s \varphi(s)\,\mathrm{d}s.
\]
Let us define $\tilde{\varphi}(t):=\sup_{s\varphi(s)\leq t}s$, for $t\geq0$. The function $\tilde{\Phi}$ can be rewritten as follows
\[
\tilde{\Phi}(t)=\int_{0}^{t} s \tilde{\varphi}(s)\,\mathrm{d}s.
\]
The function $\tilde{\Phi}$ is called the \textit{complementary function} to $\Phi$. The set
\[
L_{\Phi}(\Omega):=\left\{u:\Omega\rightarrow\mathbb{R}:\int_\Omega \Phi\left(\frac{|u(x)|}{\lambda}\right)<\infty, \mbox{ for some } \lambda>0 \right\},
\]
is called \textit{Orlicz space}. The usual norm on $L_{\Phi}(\Omega)$ is the \textit{Luxemburg norm}
\[
\|u\|_\Phi=\inf\left\{\lambda>0~|~\int_\Omega \Phi\left(\frac{|u(x)|}{\lambda}\right) \leq 1\right\}.
\]
We recall that the \textit{Orlicz-Sobolev space} $W^{1,\Phi}(\Omega)$ is defined by
\[
W^{1,\Phi}(\Omega):=\left\{u\in L_{\Phi}(\Omega):\exists f_{i}\in L_{\Phi}(\Omega), \int_{\Omega}u\frac{\partial \phi}{\partial x_{i}}=-\int_{\Omega}f_{i}\phi, \ \forall \phi\in C^{\infty}_{0}(\Omega), \ i=1,...,N\right\}.
\]
The Orlicz-Sobolev norm of $ W^{1, \Phi}(\Omega)$ is given by
\[
\displaystyle \|u\|_{1,\Phi}=\|u\|_\Phi+\sum_{i=1}^N\left\|\frac{\partial u}{\partial x_i}\right\|_\Phi.
\]
Since $\Phi$ satisfies the $\Delta_{2}$-condition, we define by $\w$  the closure of $C_0^{\infty}(\Omega)$ with respect to the Orlicz-Sobolev norm of $W^{1,\Phi}(\Omega)$. By the Poincar\'e Inequality (see e.g.  \cite{gossez-Czech}), that is, the inequality
\[
\int_\Omega\Phi(u)\leq \int_\Omega\Phi(2d_{\Omega}|\nabla u|),
\]
where $d_{\Omega}=\mbox{diam}(\Omega)$, we can conclude that
\[
\|u\|_\Phi\leq 2d_{\Omega}\|\nabla u\|_\Phi, \quad \mbox{for all} \hspace{0,2cm} u\in \w.
\]
As a consequence, we have that  $\|u\| :=\|\nabla u\|_\Phi$ defines a norm in $\w$ which is equivalent to $\|\cdot\|_{1,\Phi}$. The spaces $L_{\Phi}(\Omega)$, $W^{1,\Phi}(\Omega)$ and $W^{1,\Phi}_{0}(\Omega)$ are separable and reflexive when $\Phi$ and $\tilde{\Phi}$ satisfy the $\Delta_{2}$-condition. 

Recall also that $\Psi$ dominates $\Phi$ near infinity, in short we write $\Phi<\Psi$, if there exist positive constants $t_{0}$ and $k$ such that 
 \[
  \Phi(t)\leq \Psi(kt), \quad \mbox{for all} \hspace{0,2cm} t\geq t_{0}.
 \]
If $\Phi<\Psi$ and $\Psi<\Phi$, then we say that $\Phi$ and $\Psi$ are equivalent, and we denote by $\Phi\approx \Psi$.  Let $\Phi_*$ be the inverse of the function
$$
t\in(0,\infty)\mapsto\int_0^t\frac{\Phi^{-1}(s)}{s^{\frac{N+1}{N}}}ds
$$
which extends to $\mathbb{R}$ by  $\Phi_*(t)=\Phi_*(-t)$ for  $t\leq 0$. We say that $\Phi$ increases essentially more slowly than $\Psi$ near infinity, in short we write $\Phi<<\Psi$, if and only if for every positive constant $k$ one has
 \[
  \lim_{t\rightarrow \infty}\frac{\Phi(kt)}{\Psi(t)}=0.
 \]
It is important to emphasize that if $\Phi<\Psi<<\Phi_*$, then the following embedding
$$
\displaystyle \w \hookrightarrow L_\Psi(\Omega),
$$
is compact. In particular, since $\Phi<<\Phi_*$ (cf. \cite[Lemma 4.14]{Gz1}), we have that $\w$ is compactly embedded into $L_\Phi(\Omega)$.
Furthermore, we have that $W_0^{1,\Phi}(\Omega)$ is continuous embedded into $L_{\Phi_*}(\Omega)$. Finally, we recall the following Lemma due to N. Fukagai et al. \cite{Fuk_1} which can be written in the following way:
\begin{proposition}\label{lema_naru}
	Assume that $(\phi_1)-(\phi_3)$ hold and set
	$$
	\zeta_0(t)=\min\{t^\ell,t^m\}~~\mbox{and}~~ \zeta_1(t)=\max\{t^\ell,t^m\},~~ t\geq 0.
	$$
	Then $\Phi$ satisfies the following estimates:
	$$
	\zeta_0(t)\Phi(\rho)\leq\Phi(\rho t)\leq \zeta_1(t)\Phi(\rho), \quad \rho, t> 0,
	$$
	$$
	\zeta_0(\|u\|_{\Phi})\leq\int_\Omega\Phi(u)dx\leq \zeta_1(\|u\|_{\Phi}), \quad u\in L_{\Phi}(\Omega).
	$$
\end{proposition}

For the function $\Phi_*$ we obtain similar estimates given by the following result. 

\begin{proposition}\label{lema_naru_*}
	Assume that  $\phi$ satisfies $(\phi_1)-(\phi_3)$.  Set
	$$
	\zeta_2(t)=\min\{t^{\ell^*},t^{m^*}\},~~ \zeta_3(t)=\max\{t^{\ell^*},t^{m^*}\},~~  t\geq 0
	$$
	where $1<\ell,m<N$ and $m^* = \frac{mN}{N-m}$, $\ell^* = \frac{\ell N}{N-\ell}$.  Then
	$$
	\ell^*\leq\frac{t\Phi'_*(t)}{\Phi_*(t)}\leq m^*,~t>0,
	$$
	$$
	\zeta_2(t)\Phi_*(\rho)\leq\Phi_*(\rho t)\leq \zeta_3(t)\Phi_*(\rho),~~ \rho, t> 0,
	$$
	$$
	\zeta_2(\|u\|_{\Phi_{*}})\leq\int_\Omega\Phi_{*}(u)dx\leq \zeta_3(\|u\|_{\Phi_*}),~ u\in L_{\Phi_*}(\Omega).
	$$
\end{proposition}


\medskip
\bigskip



\bigskip

\end{document}